\documentclass[a4paper, 11pt]{amsart}
\usepackage{amsmath, amssymb, amsthm, mathrsfs}
\usepackage[mathscr]{eucal}
\usepackage[all]{xy}
\usepackage{graphicx}

\theoremstyle{plain} 
 \newtheorem{thm}{Theorem}[section]
 \newtheorem{lem}[thm]{Lemma}
 \newtheorem{cor}[thm]{Corollary}

\theoremstyle{definition}

\theoremstyle{remark}
  \newtheorem{rem}[thm]{Remark}

\newcommand{\N}{\mathbb{N}}
\newcommand{\Z}{\mathbb{Z}}
\newcommand{\calr}{\mathcal{R}}

\newcommand{\calg}{\mathcal{G}}

\newcommand{\cals}{\mathcal{S}}
\newcommand{\cali}{\mathcal{I}}

\newcommand{\calt}{\mathcal{T}}

\newcommand{\calv}{\mathcal{V}}

\newcommand{\fs}{\mathfrak{s}}
\newcommand{\gr}{\mathrm{graph}}
\newcommand{\ls}{\leqslant}
\newcommand{\relmiddle}[1]{\mathrel{}\middle#1\mathrel{}}

\renewcommand{\c}{\curvearrowright}

  \makeatletter
  \@addtoreset{equation}{section}
  \makeatother

\begin{document}

\title{Stable actions and central extensions}
\author{Yoshikata Kida}
\address{Graduate School of Mathematical Sciences, the University of Tokyo, 3-8-1 Komaba, Tokyo 153-8914, Japan}
\email{kida@ms.u-tokyo.ac.jp}
\date{March 7, 2017}
\thanks{The author was supported by JSPS Grant-in-Aid for Scientific Research, No.25800063}

\begin{abstract}
A probability-measure-preserving action of a countable group is called stable if its transformation-groupoid absorbs the ergodic hyperfinite equivalence relation of type $\textrm{II}_1$ under direct product.
We show that for a countable group $G$ and its central subgroup $C$, if $G/C$ has a stable action, then so does $G$.
Combining a previous result of the author, we obtain a characterization of a central extension having a stable action.
\end{abstract}

\maketitle


\section{Introduction}\label{sec-intro}

In his seminal paper \cite{dye1}, Dye established fundamentals of orbit equivalence relations for group actions on probability measure spaces.
Among other things, he showed uniqueness of the ergodic hyperfinite equivalence relation of type $\textrm{II}_1$, denoted by $\calr_0$.
Jones-Schmidt \cite{js} characterized when a discrete measured equivalence relation $\calr$ is \textit{stable}, that is, we have an isomorphism $\calr \times \calr_0\simeq \calr$, in terms of  asymptotically central sequences in the full group $[\calr]$, following the counterpart in the von Neumann algebra setting (\cite{mc}, \cite{connes}).
They discussed the problem asking which countable group has an ergodic, free and probability-measure-preserving (p.m.p.)\ action whose orbit equivalence relation is stable.
We call such groups \textit{stable}.
It is shown that any stable group is inner amenable (\cite[Proposition 4.1]{js}).
Recent progress around this problem are found in \cite{dv}, \cite{kec}, \cite{kida-inn}--\cite{kida-sth}, \cite{pv} and \cite{td}.
Remarkably, Tucker-Drob \cite{td} clarified structure of inner amenable groups satisfying a certain minimal condition on centralizers.
All linear groups satisfy this condition.
As its consequence, he characterized stable groups among such inner amenable groups.

Countable groups with infinite center are a typical example of inner amenable groups.
In \cite{kida-srt}, toward characterizing when a central group-extension is stable, the author obtained some necessary and sufficient conditions, which involve relative property (T) of the central subgroup.
The remaining question was the following:
Is any central extension of a stable group stable?
The aim of this paper is to answer it affirmatively through the following:

\begin{thm}\label{thm-ce}
Let $1\to C\to G\stackrel{q}{\to}\Gamma \to 1$ be an exact sequence of countable groups such that $C$ is central in $G$.
Let $\Gamma \c (X, \mu)$ be an ergodic, free, p.m.p.\ and stable action.
Then there exists an ergodic p.m.p.\ extension $\Gamma \c (Z, \zeta)$ of the action $\Gamma \c (X, \mu)$ such that if $G$ acts on $(Z, \zeta)$ through the map $q$, then the transformation-groupoid $G\ltimes (Z, \zeta)$ is stable.
\end{thm}

We say that a discrete measured groupoid $\calg$ is \textit{stable} if $\calg \times \calr_0\simeq \calg$.
We say that a p.m.p.\ action $G\c (X, \mu)$ is \textit{stable} if the transformation-groupoid $G\ltimes (X, \mu)$ is stable.
By \cite[Theorem 1.4]{kida-srt}, for a countable group to be stable, it suffices that the group has a p.m.p.\ stable action.
Hence it follows from Theorem \ref{thm-ce} that any central extension of a stable group is stable.
We note that this result is non-trivial even when the central subgroup is finite.
Theorem \ref{thm-ce} is proved by weakening influence of the 2-cocycle of the central extension, through taking extensions of the action $\Gamma \c X$.
This enables us to lift some asymptotically central sequence in $\Gamma \ltimes X$ to an asymptotically central sequence in $G\ltimes Z$.
Combining Theorem \ref{thm-ce} with \cite[Theorem 1.1]{kida-srt}, we obtain the following:

\begin{cor}\label{cor-ce}
Let $G$ be a countable group and $C$ a central subgroup of $G$.
Then $G$ is stable if and only if either the pair $(G, C)$ does not have property (T) or $G/C$ is stable.
\end{cor}

\begin{rem}
Let $\calr$ be a discrete measured equivalence relation on a standard probability space $(X, \mu)$.
We define $\calr^{(2)}$ as the set of ordered pairs $(g, h)$ of elements of $\calr$ such that their product $gh$ is defined.
Let $C$ be an abelian countable group.
We define $Z^2(\calr, C)$ as the space of \textit{2-cocycles}, i.e., measurable maps $\sigma \colon \calr^{(2)}\to C$ such that
\[\sigma(gh, k)\sigma(g, h)=\sigma(g, hk)\sigma(h, k)\ {\rm for\ any}\ g, h, k\in \calr \ {\rm with}\ (g, h), (h, k)\in \calr^{(2)}.\]
It is natural to ask the following question from the viewpoint to generalize Theorem \ref{thm-ce} in the framework of 2-cocycles of $\calr$:
Suppose that $\calr$ is p.m.p.\ and stable.
Fix $\sigma \in Z^2(\calr, C)$.
Does there exist a p.m.p.\ action $\calr \c (Z, \zeta)$ such that if $\tilde{\calr}$ denotes the equivalence relation on $(Z, \zeta)$ for the action, then the central extension of $\tilde{\calr}$ associated with $\sigma$ is stable?
We note that $Z^2(\calr, C)$ is naturally a subspace of $Z^2(\tilde{\calr}, C)$ through the projection from $Z$ onto $X$ and that each element of $Z^2(\tilde{\calr}, C)$ associates a central extension of $\tilde{\calr}$ as well as that for groups (\cite{series}).
Theorem \ref{thm-ce} answers this question affirmatively when $\calr$ arises from an ergodic, free and p.m.p.\ action of a countable group $\Gamma$ and $\sigma$ is induced by a central extension of $\Gamma$ by $C$.
These two assumptions are needed in Lemma \ref{lem-skn}, which is crucial in our construction.
More details are discussed in Remark \ref{rem-2}.
It is remarkable that there is such a technical obstacle to obtain the above generalization of Theorem \ref{thm-ce}.
\end{rem}

\begin{rem}
Let $G$ be a countable group and $N$ a normal subgroup of $G$.
Suppose that the action $G\ltimes N\c N$ is amenable, i.e., it has an invariant mean, where $G$ acts on $N$ by conjugation and $N$ acts on itself by translation.
The proof of \cite[Theorem 13 (viii)]{td} shows that any conjugacy-invariant mean on $G/N$ lifts to a conjugacy-invariant mean on $G$.
It is interesting to explore relationship among stability of $G$ and $G/N$ and property (T) of the pair $(G, N)$, toward generalizing Corollary \ref{cor-ce}. 
\end{rem}

In Section \ref{sec-pre}, we fix the notation and terminology on discrete measured groupoids, and review Jones-Schmidt's characterization of stability and co-induced actions for equivalence relations.
In Section \ref{sec-proof}, we prove Theorem \ref{thm-ce}.
Throughout the paper, all relations among measurable sets and maps that appear in the paper are understood to hold up to sets of measure zero, and all probability measure spaces are standard, unless otherwise stated.

The author thanks Narutaka Ozawa, Yuhei Suzuki and Yoshimichi Ueda for stimulus conversations around this work, and thanks the referee for reading the manuscript carefully and providing valuable suggestions.


\section{Preliminaries}\label{sec-pre}

\subsection{Stability and Jones-Schmidt's characterization}

Let $(X, \mu)$ be a standard probability space.
Let $G$ be a countable group and $G\c (X, \mu)$ its p.m.p.\ action.
We denote by $G \ltimes (X, \mu)$ the \textit{transformation-groupoid} associated with this action, and simply denote by $G \ltimes X$ if there is no confusion.
This is a discrete measured groupoid on $(X, \mu)$ whose elements are exactly those of $G\times X$.
The range and source of $(\gamma, x)\in G \ltimes X$ are $\gamma x$ and $x$, respectively.
The product on $G \ltimes X$ is defined by $(\gamma, \delta x)(\delta, x)=(\gamma \delta, x)$ for $\gamma, \delta \in G$ and $x\in X$.
The unit at $x\in X$ is $e_x:=(e, x)$.
The inverse of $(\gamma, x)\in G\ltimes X$ is $(\gamma^{-1}, \gamma x)$.
For a measurable map $U\colon X\to G$, let $U^\circ \colon X\to X$ denote the map defined by $U^\circ(x)=U(x)x$ for $x\in X$.
We define the {\it full group} of $\calg$, denoted by $[\calg]$, as the set of measurable maps $U\colon X\to G$ such that $U^\circ$ is an automorphism of $X$.
For $U, V\in [\calg]$, we define their product $U\cdot V\in [\calg]$ by $(U\cdot V)(x)=U(V^\circ(x))V(x)$ for $x\in X$.
For $U\in [\calg]$, we define its inverse $U^\dashv \in [\calg]$ by $U^\dashv(x)=U((U^\circ)^{-1}(x))^{-1}$ for $x\in X$.
The set $[\calg]$ has the group structure under these operations such that the neutral element is the map sending any $x\in X$ to $e$.

Let $\calr$ be a discrete measured equivalence relation on $(X, \mu)$.
This is a groupoid on $X$ as follows:
The range and source of $(x, y)\in \calr$ are $x$ and $y$, respectively; the product of $(x, y), (y, z)\in \calr$ is $(x, z)$; the unit at $x\in X$ is $e_x:=(x, x)$; and the inverse of $(x, y)\in \calr$ is $(y, x)$.
Let $\tilde{\mu}$ denote the measure on $\calr$ defined so that for a measurable subset $S\subset \calr$, the measure $\tilde{\mu}(S)$ is the $\mu$-integral of the function on $X$, $x\mapsto \vert \{ \, y\in X\mid (x, y)\in S \, \} \vert$.
We define the {\it full group} of $\calr$, denoted by $[\calr]$, as the group of automorphisms of $X$ whose graphs are contained in $\calr$.
Given a p.m.p.\ action of a countable group, $G\c (X, \mu)$, we denote by $\calr(G\c (X, \mu))$ the orbit equivalence relation of the action, i.e.,
\[\calr(G\c (X, \mu))=\{ \, (\gamma x, x)\in X\times X\mid \gamma \in G,\ x\in X\, \}.\]
It is denoted by $\calr(G\c X)$ if $\mu$ is understood from the context.

A sequence of measurable subsets of $X$, $(A_n)_n$, is called {\it asymptotically invariant (a.i.)} for the groupoid $\calg =G\ltimes X$ if for any $g\in G$, we have $\mu(gA_n\bigtriangleup A_n)\to 0$ as $n\to \infty$.
An a.i.\ sequence $(A_n)_n$ for $\calg$ is called {\it non-trivial} if $\mu(A_n)$ is uniformly away from $0$ and $1$.

\begin{thm}[\cite{js}, \cite{kida-srt}]\label{thm-js}
Let $G$ be a countable group and $G\c (X, \mu)$ an ergodic p.m.p.\ action.
We set $\calg =G\ltimes (X, \mu)$.
Let $\calr_0$ be the ergodic hyperfinite equivalence relation of type ${\rm II}_1$.
Then $\calg$ is stable, i.e., we have an isomorphism $\calg \simeq \calg \times \calr_0$, if and only if there exists a sequence of elements of $[\calg]$, $(U_n)_n$, satisfying the following conditions {\rm (a)--(c)}:
\begin{enumerate}
\item[(a)] For any measurable subset $A\subset X$, we have $\mu(U_n^\circ A\bigtriangleup A)\to 0$ as $n\to \infty$.
\item[(b)] For any $V\in [\calg]$, we have $\mu(\{\, x\in X\mid (U_n\cdot V)x\neq (V\cdot U_n)x\,\})\to 0$ as $n\to \infty$.
\item[(c)] There exists a non-trivial a.i.\ sequence $(A_n)_n$ for $\calg$ such that $\mu(U_n^\circ A_n\bigtriangleup A_n)$ does not converge to $0$ as $n\to \infty$.
\end{enumerate}
\end{thm}

This was originally proved by Jones-Schmidt \cite[Theorem 3.4]{js} for free actions, and the generalization to non-free actions was proved in \cite[Theorem 3.1]{kida-srt}.
We call a sequence of elements of $[\calg]$, $(U_n)_n$, satisfying the above conditions (a)--(c) a \textit{non-trivial asymptotically central (a.c.) sequence} for $\calg$.
For $V\in [\calg]$, let us say that $(U_n)_n$ \textit{asymptotically commutes} with $V$ if $\mu(\{\, x\in X\mid (U_n\cdot V)x\neq (V\cdot U_n)x\,\})\to 0$ as $n\to \infty$.

\begin{rem}\label{rem-1/2}
With the notation in Theorem \ref{thm-js}, we suppose that the action $G\c (X, \mu)$ is p.m.p.\ and not necessarily ergodic and that there exists a sequence $(U_n)_n$ in $[\calg]$ satisfying conditions (a), (b) and the following:
\begin{enumerate}
\item[${\rm (c)}'$] There exists an a.i.\ sequence $(A_n)_n$ for $\calg$ such that $U_n^\circ A_n=X\setminus A_n$ for any $n$.
\end{enumerate}
Let $\theta \colon (X, \mu)\to (T, \xi)$ be the ergodic decomposition for the action $G\c (X, \mu)$.
We have $\mu =\int_T\mu_t\, d\xi(t)$, the disintegration of $\mu$ with respect to $\theta$.
It follows that for $\xi$-a.e.\ $t\in T$, the groupoid $G\ltimes (X, \mu_t)$ is p.m.p.\ and ergodic.
For any measurable subset $A\subset X$ and any $V\in [\calg]$, the equation $\mu(V^\circ A\bigtriangleup A)=\int_T\mu_t(V^\circ A\bigtriangleup A)\, d\xi(t)$ holds.
Using this, we can find a subsequence of $(U_n)_n$ such that for $\xi$-a.e.\ $t\in T$, the subsequence is a non-trivial a.c.\ sequence for $G\ltimes (X, \mu_t)$.
For $\xi$-a.e.\ $t\in T$, the groupoid $G\ltimes (X, \mu_t)$ is therefore stable. 
\end{rem}

For a subset $\calv \subset [\calg]$, let $\calv^\dashv$ denote the set of inverses of elements in $\calv$ and set
\[\bar{\calv}=\{ \, V_1\cdot \cdots \cdot V_m\in [\calg]\mid V_1,\ldots, V_m\in \calv \cup \calv^\dashv,\ m\in \N \, \}.\]
We say that $\calv$ \textit{generates} $\calg$ if the equation $\calg =\{ \, (Vx, x)\in \calg \mid V\in \bar{\calv},\ x\in X\, \}$ holds.

As shown in Lemma \ref{lem-ab} (i) below, in Theorem \ref{thm-js}, condition (a) is indeed superfluous if condition (b) is assumed.
In our construction of stable actions, we will obtain condition (a) and a condition weaker than condition (b), and derive condition (b) from them, applying Lemma \ref{lem-ab} (ii).

\begin{lem}\label{lem-ab}
Let $G$ be a countable group and $G\c (X, \mu)$ a p.m.p.\ action.
We set $\calg =G\ltimes (X, \mu)$.
Let $(U_n)_n$ be a sequence of elements of $[\calg]$.
Regarding conditions {\rm (a)} and {\rm (b)} in Theorem \ref{thm-js} for $(U_n)_n$, the following assertions hold:
\begin{enumerate}
\item[(i)] If the action $G\c (X, \mu)$ is ergodic and not essentially transitive, then condition {\rm (b)} implies condition {\rm (a)}.
\item[(ii)] If condition {\rm (a)} holds and there is a subset $\calv \subset [\calg]$ generating $\calg$ such that $(U_n)_n$ asymptotically commutes with any element of $\calv$, then condition {\rm (b)} holds.
\end{enumerate}
\end{lem}

This lemma indeed holds for more general groupoids, not only transformation-groupoids, and almost the same proof is available.
To avoid complication, we do not deal with general groupoids.

\begin{proof}[Proof of Lemma \ref{lem-ab}]
We prove assertion (i).
Let $A\subset X$ be a measurable subset.
There exists $V\in [\calg]$ such that $V=e$ on $A$ and $V\neq e$ on $X\setminus A$.
In fact, dividing $X\setminus A$ into two measurable subsets of the same measure, we can find $V\in [\calg]$ such that $V=e$ on $A$ and $V^\circ$ exchanges the two subsets, by the assumption in assertion (i).
For any $x\in A$, we have $(U_n\cdot V)x=U_nx$ and $(V\cdot U_n)x=V(U_n^\circ x)U_nx$.
By condition (b), if $n$ is large enough, then for any $x\in A$ outside a subset of small measure, these two elements of $G$ are equal and thus $V(U_n^\circ x)=e$.
By the definition of $V$, for any such $x$, we have $U_n^\circ x\in A$.
The measure $\mu(U_n^\circ A\bigtriangleup A)$ is therefore small.
Assertion (i) follows.

To prove assertion (ii), we pick $V\in [\calg]$ and $\varepsilon >0$.
Since $\calv$ generates $\calg$, there exist a countable subset $\mathcal{W}\subset \bar{\calv}$ and a decomposition into measurable subsets, $X=\bigsqcup_{v\in \mathcal{W}}A_v$, such that for each $v\in \mathcal{W}$, we have $V=v$ on $A_v$.
There exists a finite subset $F\subset \mathcal{W}$ with $\mu(X\setminus \bigsqcup_{v\in F}A_v)<\varepsilon$.
By the assumption on $(U_n)_n$, for any large enough $n$ and for any $v\in F$, we have $\mu(U_n^\circ A_v\bigtriangleup A_v)<\varepsilon /|F|$ and $\mu(\{ \, U_n\cdot v\neq v\cdot U_n\, \})<\varepsilon /|F|$.
For any $x\in A_v$ with $(U_n\cdot V)x\neq (V\cdot U_n)x$, either $U_n^\circ x\not\in A_v$ or $U_n^\circ x\in A_v$ and $(U_n\cdot v)x=(U_n\cdot V)x\neq (V\cdot U_n)x=(v\cdot U_n)x$.
It follows that $\mu(\{ \, U_n\cdot V\neq V\cdot U_n\, \} \cap A_v)<2\varepsilon /|F|$ for any $v\in F$ and that $\mu(\{ \, U_n\cdot V\neq V\cdot U_n\, \})<3\varepsilon$.
Condition (b) follows.
\end{proof}

When $\calg$ arises from a central group-extension, condition (b) in Theorem \ref{thm-js} is further reduced as follows:

\begin{lem}\label{lem-cent}
Let $1\to C\to G\stackrel{q}{\to}\Gamma \to 1$ be an exact sequence of countable groups such that $C$ is central in $G$.
Let $\Gamma \c (X, \mu)$ be a free and p.m.p.\ action.
Let $G$ act on $(X, \mu)$ through $q$.
Set $\calr =\calr(\Gamma \c X)$ and $\calg =G\ltimes X$.
We define the quotient map $\tilde{q}\colon \calg \to \calr$ by $\tilde{q}(\gamma, x)=(q(\gamma)x, x)$ for $\gamma \in G$ and $x\in X$.
Pick a measurable section $\fs \colon \calr \to \calg$ of $\tilde{q}$.

For $V\in [\calr]$, let $\tilde{V}\colon X\to G$ denote the element of $[\calg]$ defined by $(\tilde{V}(x), x)=\fs(Vx, x)$ for $x\in X$.
Let $(U_k)_k$ be a sequence in $[\calg]$ satisfying the following conditions {\rm (a)} and {\rm (b)}:
\begin{enumerate}
\item[(a)] For any measurable subset $A\subset X$, we have $\mu(U_k^\circ A\bigtriangleup A)\to 0$ as $k\to \infty$.
\item[(b)] There exists a subset $\calv \subset [\calr]$ generating $\calr$ such that for any $V\in \calv$, the sequence $(U_k)_k$ asymptotically commutes with $\tilde{V}$.
\end{enumerate}
Then the sequence $(U_k)_k$ asymptotically commutes with any element of $[\calg]$.
\end{lem}

\begin{proof}
We define $\mathcal{W}$ as the set of elements $W\in [\calg]$ such that there are an element $V\in \calv$ and a measurable map $\varphi \colon X\to C$ with $W(x)=\tilde{V}(x)\varphi(x)$ for any $x\in X$.
The set $\mathcal{W}$ then generates $\calg$.
By Lemma \ref{lem-ab} (ii), it suffices to show that $(U_k)_k$ asymptotically commutes with any element of $\mathcal{W}$. 
Pick $W\in \mathcal{W}$ and take $V\in \calv$ and $\varphi \colon X\to C$ as above.
Fix $c\in C$ and set $A=\varphi^{-1}(c)$.
Conditions (a) and (b) imply that if $k$ is large enough, then for any $x\in A$ outside a subset of small measure, we have $U_k^\circ x\in A$ and $(U_k\cdot \tilde{V})x=(\tilde{V}\cdot U_k)x$, and hence $(U_k\cdot W)x=U_k(W^\circ x)W(x)=U_k(\tilde{V}^\circ x)c\tilde{V}(x)= c\tilde{V}(U_k^\circ x)U_k(x) =W(U_k^\circ x)U_k(x)=(W\cdot U_k)x$, where the fourth equation holds because $U_k^\circ x\in A$.
The lemma follows.
\end{proof}


\subsection{Co-induced actions for equivalence relations}

Let $(X, \mu)$ be a standard probability space.
Let $\calr$ be a discrete measured equivalence relation on $(X, \mu)$ which is \textit{p.m.p.}, that is, any element of $[\calr]$ preserves the measure $\mu$.
Let $\cals$ be a subrelation of $\calr$.
Suppose that $\cals$ has a \textit{p.m.p.\ action} on a standard probability space $(B, \lambda)$ in the following sense:
We have a measurable map from $\cals \times B$ into $B$, denoted by $(g, b)\mapsto gb$, such that $e_xb=b$ for any $x\in X$ and any $b\in B$; $g(hb)=(gh)b$ for any $g, h\in \cals$ whose product $gh\in \cals$ is defined and for any $b\in B$; and for any $g\in \cals$, the automorphism of $B$, $b\mapsto gb$, preserves the measure $\lambda$.
The aim of this subsection is to introduce the action of $\calr$ co-induced from the action $\cals \c (B, \lambda)$, analogously to co-induced actions for countable groups (e.g., \cite[\S 3.4]{g} and \cite[\S 10 (G)]{kec}).

An action of a groupoid is usually defined as an action on a fibered space over the unit space which preserves fibers in the way compatible with the range and source maps.
The action $\cals \c (B, \lambda)$ defined above corresponds to a groupoid action on a constant fibered space.
For our purpose and simplifying construction, we are only concerned with the action co-induced from such a special action of $\cals$.

For $x\in X$, let $\calr^x$ be the set of elements of $\calr$ whose ranges are $x$, and define the set
\[\Sigma_x=\{ \, f\colon \calr^x\to B\mid f(x, y)=(y, z)f(x, z)\ {\rm for\ any}\ y, z\in \calr x\ {\rm with}\ (y, z)\in \cals \, \},\]
where for $x\in X$, we denote by $\calr x$ the $\calr$-equivalence class containing $x$.
Set $\Sigma =\bigsqcup_{x\in X}\Sigma_x$.
We have the specified projection from $\Sigma$ onto $X$ sending each element of $\Sigma_x$ to $x$.
Let $\calr$ act on $\Sigma$ as follows:
For $g=(y, x)\in \calr$ and $f\in \Sigma_x$, we define $gf\in \Sigma_y$ by $(gf)(y, z)=f(g^{-1}(y, z))=f(x, z)$ for $z\in \calr y$.

We equip $\Sigma$ with measure-space structure.
Let $I\colon X\to \N \cup \{ \infty \}$ be the \textit{index function} for the inclusion $\cals <\calr$, namely, the function assigning to each $x\in X$ the number of $\cals$-equivalence classes contained in $\calr x$.
By \cite[Lemma 1.1 (a)]{fsz} and its proof, the function $I$ is measurable and for any $(x, y)\in \calr$, we have $I(x)=I(y)$.

We first suppose that $I$ is constant and its value is $N\in \N \cup \{ \infty \}$.
By \cite[Lemma 1.1 (b)]{fsz}, there exists a family of measurable maps from $X$ into itself, $\{ \, \phi_i\mid 0\leq i<N\, \}$, such that $\phi_0$ is the identity map of $X$ and the family $\{ \, \cals \phi_i(x)\mid 0\leq i<N\, \}$ partitions $\calr x$.
We have the bijection $\Phi \colon \Sigma \to X\times \prod_{0\leq i<N}B$ defined by $\Phi(f)=(x, (f(x, \phi_i(x)))_i)$ for $x\in X$ and $f\in \Sigma_x$.
Under this identification, $\Sigma$ is equipped with the measurable structure induced by the product measurable structure on $X\times \prod_{0\leq i<N}B$.
It is also equipped with the probability measure induced by the product measure $\mu \times \prod_{0\leq i<N}\lambda$.
This structure on $\Sigma$ does not depend on the choice of the family $\{ \phi_i\}_i$.
The action of $\calr$ on $\Sigma$ is measurable and p.m.p. We call it the action of $\calr$ \textit{co-induced} from the action $\cals \c (B, \lambda)$.

A remarkable property of the co-induced action is that the projection $p\colon \Sigma \to B$ defined by $p(f)=f(e_x)$ for $f\in \Sigma_x$ and $x\in X$ is $\cals$-equivariant.
That is, for any $g=(y, x)\in \cals$ and $f\in \Sigma_x$, we have $p(gf)=gp(f)$.
This equation follows because $p(gf)=(gf)(e_y)=f(g^{-1}e_y)=f(x, y)=f(e_xg^{-1})=g(f(e_x))=gp(f)$.

If $I$ is not constant, then $X$ is decomposed into $\calr$-invariant subsets $X_N:=I^{-1}(N)$ with $N\in \N \cup \{ \infty \}$.
We have the action of $\calr|_{X_N}\c \Sigma_N$ co-induced from the action $\cals |_{X_N}\c B$, where we set $\calr|_{X_N}=\calr \cap (X_N\times X_N)$ and set $\cals|_{X_N}$ similarly.
We call the union of these actions, $\calr \c \bigsqcup_N \Sigma_N$, the action of $\calr$ \textit{co-induced} from the action $\cals \c (B, \lambda)$.
As shown in the last paragraph, we have the projection from $\bigsqcup_N\Sigma_N$ onto $B$ that is $\cals$-equivariant.


\section{Proof of Theorem \ref{thm-ce}}\label{sec-proof}

Let $1\to C\to G\stackrel{q}{\to}\Gamma \to 1$ be an exact sequence of countable groups such that $C$ is central in $G$.
Suppose that we have a free, p.m.p.\ and stable action $\Gamma \c (X, \mu)$.
We will construct a p.m.p.\ extension $\Gamma \c (\Omega, \eta)$ of the action $\Gamma \c (X, \mu)$ such that if $G$ acts on $(\Omega, \eta)$ through the map $q$, then the groupoid $G\ltimes (\Omega, \eta)$ has a non-trivial a.c.\ sequence.
Ergodicity of the action $\Gamma \c (X, \mu)$ will be assumed after Lemma \ref{lem-c-infinite} to prove Theorem \ref{thm-ce} in the end of this section.

Let $\mathsf{s}\colon \Gamma \to G$ be a section of the quotient map $q\colon G\to \Gamma$ with $\mathsf{s}(e)=e$.
We have the 2-cocycle $\sigma \colon \Gamma \times \Gamma \to C$ associated with $\mathsf{s}$.
Namely, it is defined by $\sigma(\gamma, \delta)\mathsf{s}(\gamma \delta)=\mathsf{s}(\gamma)\mathsf{s}(\delta)$ for $\gamma, \delta \in \Gamma$.
We set $\calr =\calr(\Gamma \c X)$.
The assumption that $\calr$ is stable implies that it is decomposed as $\calr =\calr_0\times \calr_1$, where $\calr_0$ is the ergodic hyperfinite equivalence relation of type $\textrm{II}_1$.
Let $(X_0, \mu_0)$ and $(X_1, \mu_1)$ denote the probability spaces on which $\calr_0$ and $\calr_1$ are defined, respectively.
We may assume that $\calr_0 =\calr( \bigoplus_\N \Z /2\Z \c \prod_\N \Z /2\Z)$, where the action is given by addition, we have $X_0=\prod_\N \Z /2\Z$, and $\mu_0$ is the Haar measure on the compact group $\prod_\N \Z /2\Z$.
Let $\cali_0$ denote the trivial equivalence relation on $(X_0, \mu_0)$.
We naturally identify the full group $[\calr_1]$ with the subgroup of the full group $[\cali_0 \times \calr_1]$ under the map $V\mapsto \textrm{id}_{X_0} \times V$ for $V\in [\calr_1]$.
Similarly, we define $\cali_1$ and identify $[\calr_0]$ with the subgroup of $[\calr_0\times \cali_1]$.
For $n\in \N$, let $V_n\in [\calr_0]$ be the transformation exchanging $0$ and $1$ in the $n$-th coordinate of $X_0=\prod_\N \Z /2\Z$ and fixing the other coordinates.

Fix $k, n\in \N$ with $k<n$.
We set
\[X_0^{\leqslant k}=\prod_1^k\Z /2\Z\quad \textrm{and}\quad X_0^{>k}=\prod_{k+1}^\infty \Z /2\Z\]
so that $X_0=X_0^{\leqslant k}\times X_0^{>k}$ naturally.
Let $\mu_0^{\leqslant k}$ and $\mu_0^{>k}$ denote the measures on $X_0^{\leqslant k}$ and $X_0^{>k}$, respectively, with $\mu_0=\mu_0^{\leqslant k}\times \mu_0^{>k}$.
We set
\[\calr_0^{\leqslant k}=\calr \left(\bigoplus_1^k\Z /2\Z \c X_0^{\leqslant k}\right),\]
and let $\cali_0^{>k}$ be the trivial equivalence relation on $X_0^{>k}$.
We define
\[\cals_{k, n}=\{ \, h\in \calr_0^{\leqslant k}\times \cali_0^{>k}\times \calr_1\mid \tau(h)V_n z=V_n\tau(h)z,\ \textrm{where}\ z\ \textrm{is\ the\ source\ of\ }h.\, \},\] 
where $\tau \colon \calr \to \Gamma$ is the map defined by $\tau(\gamma x, x)=\gamma$ for $\gamma \in \Gamma$ and $x\in X$.
The set $\cals_{k, n}$ is a measurable subset of $\calr$.
In Lemma \ref{lem-skn} below, we will show that $\cals_{k, n}$ is an equivalence relation.
For $x\in X_0^{>k}$, let $e_x=(x, x)\in \cali_0^{>k}$ be the unit element, and we set
\[\cals_{k, n}^x=\cals_{k, n}\cap (\calr_0^{\leqslant k}\times \{ e_x\}\times \calr_1).\]
Note that $\cals_{k, n}$ consists of pairwise disjoint sets $\cals_{k, n}^x$ running through $x\in X_0^{>k}$.

Let $G$ act on $X$ through the quotient map $q\colon G\to \Gamma$, and let $G\ltimes X$ be the associated transformation-groupoid.
Using the section $\mathsf{s}\colon \Gamma \to G$, we define a section $\fs \colon \calr \to G\ltimes X$ as follows:
For $(\gamma x, x)\in \calr$ with $\gamma \in \Gamma$ and $x\in X$, we set $\fs(\gamma x, x)=(\mathsf{s}(\gamma), x)$.
We define a map $\bar{\fs}\colon \calr \to G$ as the composition of the map $\fs$ and the projection from $G\ltimes X$ onto $G$.

Fix $k, n\in \N$ with $k<n$ again.
We often naturally identify the two equivalence relations, $\calr_0^{\ls k}\times \cali_0^{>k} \times \calr_1$ and $\cali_0^{>k} \times (\calr_0^{\ls k}\times \calr_1)$, under exchanging coordinates.
We set
\[X^{\ls k}=X_0^{\ls k}\times X_1,\quad \mu^{\ls k}=\mu_0^{\ls k}\times \mu_1\quad \textrm{and}\quad \calr^{\ls k}=\calr_0^{\ls k}\times \calr_1.\]
We define a map $\sigma_{k, n}\colon \cali_0^{>k} \times \calr^{\ls k}\to C$ so that for $x\in X_0^{>k}$, we have
\[\sigma_{k, n}(e_x, h)=\bar{\fs}(e_x, h)^{-1}\bar{\fs}((V_nx, x), e_{y'})^{-1}\bar{\fs}(e_{V_nx}, h)\bar{\fs}((V_nx, x), e_y)\]
for $h=(y', y)\in \calr^{\ls k}$, where $V_n$ is regarded as an element of the full group $[\calr_0^{>k}]$ because $k<n$.
The map $\sigma_{k, n}$ detects non-commutativity of the lift of the diagram drawn in Figure \ref{fig-sq} (a) to $G\ltimes X$.
\begin{figure}
\begin{center}
\includegraphics[width=14cm]{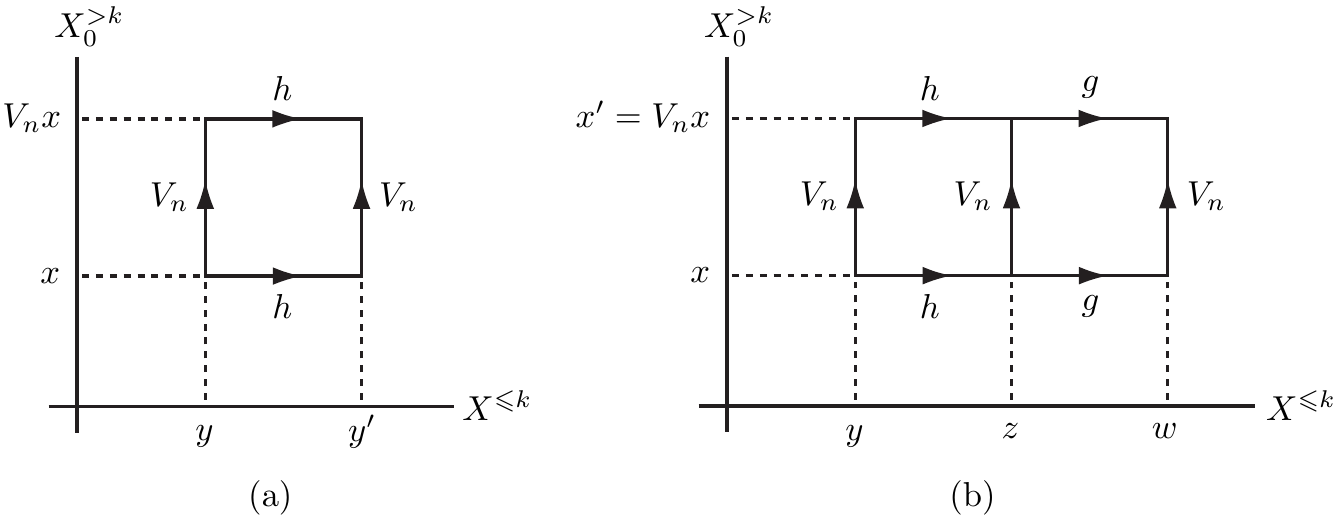}
\caption{}\label{fig-sq}
\end{center}
\end{figure}

\begin{lem}\label{lem-skn}
With the above notation, the following assertions hold:
\begin{enumerate}
\item[(i)] The set $\cals_{k, n}$ is a subrelation of $\cali_0^{>k}\times \calr^{\ls k}$.
\item[(ii)] For a.e.\ $x\in X_0^{>k}$, the subrelation $\cals_{k, n}^x$ approaches $\{ e_x\} \times \calr^{\ls k}$ as $n\to \infty$, namely, for any measurable subset $S\subset \{ e_x\} \times \calr^{\ls k}$ of finite measure, the measure of the set $S\setminus \cals_{k, n}^x$ converges to $0$ as $n\to \infty$. 
\item[(iii)] The restriction $\sigma_{k, n}\colon \cals_{k, n}\to C$ is a 1-cocycle, and for a.e.\ $x\in X_0^{>k}$, the restriction $\sigma_{k, n}\colon \cals_{k, n}^x\to C$ is a 1-cocycle.
\end{enumerate}
\end{lem}

\begin{proof}
Fix $x\in X_0^{>k}$.
For the ease of the symbol, we put $\cals =\cals_{k, n}^x$.
For any $z\in \{ x\} \times X^{\ls k}$, we have $\tau(e_z)=e$ and $e_z\in \cals$ by the definition of $\cals$.
Pick $h\in \cals$ and let $z$ denote the source of $h$.
The element $\tau(h)z$ is the range of $h$, and the equation $V_n\tau(h^{-1})(\tau(h)z)=V_nz=\tau(h^{-1})V_n(\tau(h)z)$ holds, where the last equation follows from $h\in \cals$.
This implies $h^{-1}\in \cals$.
Pick an element $g\in \cals$ whose source is $\tau(h)z$, the range of $h$.
The equation $\tau(gh)V_nz=\tau(g)V_n\tau(h)z=V_n\tau(gh)z$ holds.
This implies $gh\in \cals$.
We showed that $\cals$ is an equivalence relation on the set $\{ x\} \times X^{\ls k}$, and assertion (i) follows.

Since the sequence $(V_n)_n$ is a.c.\ for $\calr$, for any $\gamma \in \Gamma$, the measure of the set of points $z\in X$ with $\gamma V_nz=V_n\gamma z$ approaches $1$.
Assertion (ii) follows.

We prove assertion (iii).
Figure \ref{fig-sq} (b) will be helpful.
Fix $x\in X_0^{>k}$.
Pick $g, h\in \calr^{\ls k}$ so that $g=(w, z)$ and $h=(z, y)$ with $y, z, w\in X^{\ls k}$.
Suppose that $(e_x, g)$ and $(e_x, h)$ are in $\cals_{k, n}^x$.
We set $x'=V_nx\in X_0^{>k}$, $\gamma =\tau(e_x, g)$ and $\delta =\tau(e_x, h)$.
It follows from $(e_x, h)\in \cals_{k, n}^x$ that $\delta(x', y)=\delta V_n(x, y)=V_n\delta(x, y)=V_n(x, z)=(x', z)$ and hence $\tau(e_{x'}, h)=\delta$.
It similarly follows from $(e_x, g)\in \cals_{k, n}^x$ that $\tau(e_{x'}, g)=\gamma$.
By the definition of the 2-cocycle $\sigma \colon \Gamma \times \Gamma \to C$, we have $\sigma(\gamma, \delta)\mathsf{s}(\gamma \delta)=\mathsf{s}(\gamma)\mathsf{s}(\delta)$.
By the definition of the section $\fs\colon \calr \to G\ltimes X$, we have $\sigma(\gamma, \delta)\fs(e_x, gh)=\fs(e_x, g)\fs(e_x, h)$ and $\sigma(\gamma, \delta)\fs(e_{x'}, gh)=\fs(e_{x'}, g)\fs(e_{x'}, h)$ because $\tau(e_{x'}, g)=\gamma$ and $\tau(e_{x'}, h)=\delta$.
It follows that
\begin{align*}
\sigma_{k, n}(e_x, gh)&=\bar{\fs}(e_x, gh)^{-1}\bar{\fs}((x', x), e_w)^{-1}\bar{\fs}(e_{x'}, gh)\bar{\fs}((x', x), e_y)\\
&=\sigma(\gamma, \delta)\bar{\fs}(e_x, h)^{-1}\bar{\fs}(e_x, g)^{-1}\bar{\fs}((x', x), e_w)^{-1}\bar{\fs}(e_{x'}, gh)\bar{\fs}((x', x), e_y)\\
&=\bar{\fs}(e_x, h)^{-1}\bar{\fs}(e_x, g)^{-1}\bar{\fs}((x', x), e_w)^{-1}\bar{\fs}(e_{x'}, g)\bar{\fs}(e_{x'}, h)\bar{\fs}((x', x), e_y)\\
&=\bar{\fs}(e_x, h)^{-1}\sigma_{k, n}(e_x, g)\bar{\fs}((x', x), e_z)^{-1}\bar{\fs}(e_{x'}, h)\bar{\fs}((x', x), e_y)\\
&=\sigma_{k, n}(e_x, g)\sigma_{k, n}(e_x, h).
\end{align*}
The former assertion in assertion (iii) follows.

There exists a countable subset $\calv \subset [\cals_{k, n}]$ such that $\bigcup_{V\in \calv}\gr(V)=\cals_{k, n}$, where we set $\gr(V)=\{ \, (Vz, z)\in \calr \mid z\in X\, \}$ for $V\in [\calr]$.
The former assertion implies that for a.e.\ $z\in X$, for any $U, V\in \calv$, the equation $\sigma_{k, n}(UVz, Vz)\sigma_{k, n}(Vz, z)=\sigma_{k, n}(UVz, z)$ holds.
By Fubini's theorem, for a.e.\ $x\in X_0^{>k}$, for a.e.\ $y\in X^{\ls k}$ and for any $U, V\in \calv$, the same equation holds with $z=(x, y)$.
The latter assertion in assertion (iii) follows.
\end{proof}

\begin{rem}\label{rem-2}
Let $\rho \colon \calr^{(2)}\to C$ be the 2-cocycle associated with the section $\fs \colon \calr \to G\ltimes X$.
In the proof of Lemma \ref{lem-skn} (iii), we implicitly use the following property of $\rho$:
There exists a countable subset $\mathcal{U}\subset [\calr]$ such that it generates a subrelation of $\calr$ containing $\cali_0^{>k}\times \calr^{\ls k}$ and for any $U, V\in \mathcal{U}$ and any $z=(x, y)\in X=X_0^{>k}\times X^{\ls k}$ with $(Vz, z), (UVz, Vz)\in \cali_0^{>k}\times \calr^{\ls k}$, the element $\rho((UVz, Vz), (Vz, z))\in C$ only depends on $U$, $V$ and $y$.
It is not obvious whether we can realize this for general 2-cocycles in $Z^2(\calr, C)$ or realize another property for them so that the restriction of $\sigma_{k, n}$ to some subrelation of $\cali_0^{>k}\times \calr^{\ls k}$, which should approach $\cali_0^{>k}\times \calr^{\ls k}$, is a 1-cocycle.
We also note that the condition that $\calr$ arises from a free action of a countable group is essential to define the subrelation $\cals_{k, n}$ with such a desired property.   
\end{rem}

We fix a decreasing sequence of positive numbers, $1>\varepsilon_1>\varepsilon_2>\cdots$ with $\varepsilon_k\to 0$ as $k\to \infty$.
We also fix a sequence of elements of $[\calr_1]$, $(T_k)_{k=1}^\infty$, such that $\calr_1=\bigcup_{k=1}^\infty\gr(T_k)$, where we define the set $\gr(T_k)$ similarly to that in the proof of Lemma \ref{lem-skn} (iii).
We will inductively find an increasing sequence of positive integers, $n(1)<n(2)<\cdots$, as follows:
We set $n(0)=0$.
For $k\in \N$, let $n(k)$ be a positive integer $n$ such that $n>n(k-1)$ and
\begin{equation}\label{gr}
\tilde{\mu}\left( \left(\bigcup_{i=1}^k(\gr(V_i)\cup \gr(T_i)) \right)\setminus \cals_{k, n} \right) <\varepsilon_k^2.
\end{equation}
Such an $n$ exists by Lemma \ref{lem-skn} (ii).
In what follows, we fix these $k$ and $n=n(k)$, and for the ease of the symbols, we drop $n$, that is, we denote $\cals_{k, n}$, $\cals_{k, n}^x$ and $\sigma_{k, n}$ by $\cals_k$, $\cals_k^x$ and $\sigma_k$, respectively.


\medskip

\noindent \textbf{Approximating the 1-cocycle $\sigma_k\colon \cals_k\to C$.} 
For $c=(b_1, \ldots, b_k, c_1,\ldots, c_k)\in C^{2k}$, we define a subset $X_c\subset X$ by
\[X_c=\{ \, z\in X\mid \forall i=1,\ldots, k,\ \sigma_k(V_iz, z)=b_i\ \textrm{and}\ \sigma_k(T_iz, z)=c_i\, \}.\]
The set $X$ is decomposed into the sets $X_c$ with $c\in C^{2k}$.
We approximate this decomposition by cylindrical subsets with respect to the product $X=X_0^{>k}\times X^{\leqslant k}$.
Let $F\subset C^{2k}$ be a finite subset such that
\begin{equation}\label{x_c}
\mu \left( \bigsqcup_{c\in F}X_c \right)>1-\varepsilon_k^2.
\end{equation}
For each $c\in F$, there exists a measurable subset $A_c\subset X$ which is a union of finitely many cylindrical subsets of $X=X_0^{>k}\times X^{\leqslant k}$ and satisfies
\begin{equation}\label{x_c-a_c}
\mu \left( X_c\bigtriangleup A_c\right) <\varepsilon_k^2 /|F|^2.
\end{equation}
By inequality (\ref{x_c}), we have
\begin{equation}\label{a_c}
\mu \left( \bigcup_{c\in F}A_c \right)>1-2\varepsilon_k^2.
\end{equation}
Let $\Delta$ be the rectangular decomposition of $X=X_0^{>k}\times X^{\leqslant k}$ generated by the sets $A_c$ with $c\in F$.
We then obtain the decomposition $X_0^{>k}=B_1\sqcup \cdots \sqcup B_m$ from $\Delta$.

We define three measurable subsets $E_1, E_2, E_3\subset X_0^{>k}$ as follows:
We set
\[E_1=\{ \, x\in X_0^{>k}\mid \mu^{\ls k}( \{ \, y\in X^{\ls k}\mid \forall i=1,\ldots, k,\ (V_iy, y), (T_iy, y)\in \bar{\cals}_k^x \, \} )>1-\varepsilon_k\, \},\]
where for $x\in X_0^{>k}$, we denote by $\bar{\cals}_k^x$ the subrelation of $\calr^{\ls k}$ with $\cals_k^x=\{ e_x\} \times \bar{\cals}_k^x$.
By inequality (\ref{gr}), we have $\mu_0^{>k}(E_1)>1-\varepsilon_k$.
For $x\in X_0^{>k}$ and a subset $S\subset X$, let us call the set $\{ \, y\in X^{\leqslant k}\mid (x, y)\in S\, \}$ the $x$-\textit{slice} of $S$.
For $c\in F$, we set
\[E_c=\{ \, x\in X_0^{>k}\mid \mu^{\leqslant k}\left( \, {\rm the}\ x\textrm{-}{\rm slice\ of}\ X_c\bigtriangleup A_c\right) <\varepsilon_k /|F|\, \} \quad \textrm{and}\quad E_2=\bigcap_{c\in F}E_c.\]
By inequality (\ref{x_c-a_c}), we have $\mu_0^{>k}(E_c)>1-\varepsilon_k/|F|$ for any $c\in F$ and hence $\mu_0^{>k}(E_2)>1-\varepsilon_k$.
Finally, we set
\[E_3=\left\{ \, x\in X_0^{>k}\relmiddle| \mu^{\ls k}\left( \, {\rm the}\ x\textrm{-}{\rm slice\ of}\ \bigcup_{c\in F}A_c\right)>1-\varepsilon_k\, \right\}.\]
By inequality (\ref{a_c}), we have $\mu_0^{>k}(E_3)>1-2\varepsilon_k$.
Set $E=E_1\cap E_2\cap E_3$.
We then have
\begin{equation}\label{e}
\mu_0^{>k}(E)>1-4\varepsilon_k.
\end{equation}

For $j=1,\ldots, m$, we will define a subrelation $\cals_k^j$ of $\calr^{\ls k}$ and a 1-cocycle $\sigma_k^j\colon \cals_k^j\to C$ as follows:
If the measure of the set $B_j\cap E$ is zero, then let $\cals_k^j$ be the trivial equivalence relation on $X^{\ls k}$, and let $\sigma_k^j$ be the trivial cocycle.
Suppose that $B_j\cap E$ has positive measure.
We pick a point $a\in B_j\cap E$ such that the restriction $\sigma_k\colon \cals_k^a\to C$ is a 1-cocycle.
Such an $a$ exists by Lemma \ref{lem-skn} (iii).
We set $\cals_k^j=\bar{\cals}_k^a$ and define the 1-cocycle $\sigma_k^j\colon \cals_k^j\to C$ by $\sigma_k^j(g)=\sigma_k(e_a, g)$ for $g\in \cals_k^j$.

\begin{lem}
For any $x\in E$, taking $j=1,\ldots, m$ with $x\in B_j$, we have
\begin{equation}\label{dx}
\mu^{\ls k}(D_x)>1-4\varepsilon_k,
\end{equation}
where we set
\begin{align*}
D_x= \{ \, y&  \in X^{\leqslant k}  \mid {\rm Putting}\ z=(x, y)\in X,\ {\rm we\ have}\ \forall i=1,\ldots, k,\\ 
&(V_iy, y), (T_iy, y)\in \cals_k^j,\ \sigma_k(V_iz, z)=\sigma_k^j(V_iy, y)\ {\rm and}\ \sigma_k(T_iz, z)=\sigma_k^j(T_iy, y),\\
& \hspace{16em}{\rm and\ have}\ z\in X_c\ {\rm for\ some}\ c\in F.\, \}.
\end{align*}
\end{lem}

\begin{proof}
Let $D_x'$ denote the set of all points $y\in X^{\ls k}$ satisfying the following three conditions:
\begin{itemize}
\item For any $i=1,\ldots, k$, we have $(V_iy, y), (T_iy, y)\in \bar{\cals}_k^a$;
\item the point $y$ belongs to neither the $a$-slice of the set $\bigcup_{c\in F}(X_c\bigtriangleup A_c)$ nor the $x$-slice of that set; and
\item the point $y$ belongs to the $a$-slice of the set $\bigcup_{c\in F}A_c$.
\end{itemize}
By the definition of the sets $E_1$, $E_2$ and $E_3$, we have $\mu^{\ls k}(D_x')>1-4\varepsilon_k$.
For the proof of the lemma, it suffices to show the inclusion $D_x'\subset D_x$.
Pick a point $y\in D_x'$.
There exists $c\in F$ such that $y$ belongs to the $a$-slice of $A_c$ and thus belongs to the $a$-slice of $X_c$.
Since $a$ and $x$ are in the same $B_j$, the points $w:=(a, y)$ and $z:=(x, y)$ are in the same rectangular piece of the decomposition $\Delta$.
It follows from $w\in A_c$ that $z\in A_c$, that is, $y$ belongs to the $x$-slice of $A_c$.
The point $y$ thus belongs to the $x$-slice of $X_c$.
It follows from $z, w\in X_c$ that $\sigma_k(V_iz, z)=\sigma_k(V_iw, w)$ for any $i=1,\ldots, k$.
By the definition of $\sigma_k^j$, we have $\sigma_k^j(V_iy, y)=\sigma_k(V_iw, w)$ and hence $\sigma_k(V_iz, z)=\sigma_k^j(V_iy, y)$.
Similar computation shows that $\sigma_k(T_iz, z)=\sigma_k^j(T_iy, y)$.
The inclusion $D_x'\subset D_x$ follows.
\end{proof}

Inequalities (\ref{e}) and (\ref{dx}) imply
\begin{equation}\label{d}
\mu(D)>(1-4\varepsilon_k)^2>1-8\varepsilon_k,
\end{equation}
where we set
\begin{align*}
D = \{ \, z & =(x, y) \in X=X_0^{>k}\times X^{\ls k} \mid {\rm Choosing}\ j=1,\ldots, m \ {\rm with}\ x\in B_j,\\
&{\rm we\ have}\ \forall i=1,\ldots, k,\ (V_iy, y), (T_iy, y)\in \cals_k^j,\ \sigma_k(V_iz, z)=\sigma_k^j(V_iy, y)\\
& {\rm and}\ \sigma_k(T_iz, z)=\sigma_k^j(T_iy, y),\ {\rm and\ have}\ z\in X_c\ {\rm for\ some}\ c\in F.\, \}.
\end{align*}


\medskip

\noindent \textbf{Construction of a desired extension.}
It is known that any abelian countable group can be embedded into a compact group with dense image.
For example, it is obtained as the Bohr compactification (\cite[\S 4.7]{f}).
We embed the central subgroup $C$ into a compact group $K$ with dense image.
Let $K$ be equipped with the Haar measure.

Fix $k\in \N$.
We also fix $j=1,\ldots, m(k)$, where $m=m(k)$ is the number of pieces in the partition $X_0^{>k}=B_1\sqcup \cdots \sqcup B_m$ arising from the rectangular decomposition $\Delta$ of $X$.
This was obtained in the process approximating the 1-cocycle $\sigma_k$.
Let this number denote by $m(k)$ to indicate the dependence on $k$.
We define the action of $\cals_k^j$ on $K$ by the formula $gb=\sigma_k^j(g)b$ for $g\in \cals_k^j$ and $b\in K$.
Let $\calr^{\ls k}\c Y_k^j$ be the action co-induced from the action $\cals_k^j\c K$.
For $y\in X^{\ls k}$, let $(\calr^{\ls k})^y$ denote the set of elements of $\calr^{\ls k}$ whose range is $y$, and we define an equivalence relation on $(\calr^{\ls k})^y$ so that two elements $g, h\in (\calr^{\ls k})^y$ are equivalent if and only if $g^{-1}h\in \cals_k^j$.
Let $(\calr^{\ls k})^y/\cals_k^j$ denote the set of equivalence classes of this equivalence relation.
The space $Y_k^j$ is written as
\[Y_k^j=\bigsqcup_{y\in X^{\ls k}}\prod_{(\calr^{\ls k})^y/\cals_k^j}K.\] 
We have the fibered product over $X^{\ls k}$, $X\times_{X^{\ls k}}Y_k^j$, equipped with the specified projection to the first coordinate in $X$.
This is naturally isomorphic to the direct product $X_0^{>k}\times Y_k^j$.
Let $\calr$ act on the space $X\times_{X^{\ls k}}Y_k^j$ through the projection from $\calr$ onto $\calr^{\ls k}$.
Namely, the action is given by $(v, g)(z, y)=(z', gy)$ for $z\in X$ and $y\in Y_k^j$ with the same projection to $X^{\ls k}$ and for $(v, g)\in \calr =\calr_0^{>k}\times \calr^{\ls k}$ whose source and range are $z$ and $z'$, respectively.

Define the fibered product over $X$:
\[Y_k=\prod_{j=1,\ldots, m(k):/X}\left( X\times_{X^{\ls k}}Y_k^j\right),\]
where the product symbol with the symbol $``/X"$ at its bottom means the fibered product of $X\times_{X^{\ls k}}Y_k^j$ over $X$ running through $j=1,\ldots, m(k)$. 
Let $\calr$ act on $Y_k$ diagonally, namely, the action $\calr \c Y_k$ is given by $g(y^j)_j=(gy^j)_j$ for $g\in \calr$ and an element $y^j\in X\times_{X^{\ls k}}Y_k^j$ whose projection to $X$ is the source of $g$.

Making $k$ run through all natural numbers, we define the fibered product over $X$:
\[\Omega =\prod_{k\in \N : /X}Y_k=\prod_{k\in \N : /X}\prod_{j=1,\ldots, m(k):/X}\left( X\times_{X^{\ls k}}Y_k^j\right).\]
Let $\eta$ be the probability measure on $\Omega$ naturally obtained through construction of fibered products.
Let $\calr$ act on $\Omega$ diagonally in the same manner as indicated in defining the action of $\calr$ on $Y_k$.
This action preserves $\eta$.
We have the p.m.p.\ action $\Gamma \c (\Omega, \eta)$ defined by the formula $\gamma \omega =(\gamma x, x)\omega$ for $\gamma \in \Gamma$ and an element $\omega \in \Omega$ whose projection to $X$ is $x$.
Let $G$ act on $(\Omega, \eta)$ through the quotient map $q\colon G\to \Gamma$, and let $\calg =G\ltimes (\Omega, \eta)$ be the associated transformation-groupoid.

Recall that we chose the section $\fs \colon \calr \to G\ltimes X$.
For $n\in \N$, we denote by $\tilde{V}_n$ the lift of $V_n\in [\calr_0\times \cali_1]$ contained in the image of $\fs$, namely, the map $\tilde{V}_n\colon X\to G$ is defined by the equation $(\tilde{V}_nz, z)=\fs(V_nz, z)$ for $z\in X$.
Unless there is no confusion, we use the same symbol $\tilde{V}_n\colon \Omega \to G$ to denote the element of $[\calg]$ obtained by composing the projection from $\Omega$ onto $X$.

In the rest of this section, we will construct a non-trivial a.c.\ sequence $(U_k)_k$ for $\calg$.
We obtain $U_k\in [\calg]$ by twisting $\tilde{V}_{n(k)}$ through appropriate multiplication by elements of $K$ depending on points of $\Omega$.
Our construction below is separated into the two cases where $C$ is finite or infinite, although that in the latter case is available for the former case.
We decided to separate them because the construction in the former case is somewhat simpler thanks to the equation $K=C$.

\medskip


\noindent \textbf{The case where $C$ is finite.}
The equation $K=C$ then holds.
Fix $k\in \N$.
We define a map $U_k\colon \Omega \to G$ as follows:
Pick $\omega \in \Omega$.
We have a unique $j=1,\ldots, m(k)$ such that the projection of $\omega$ to $X_0^{>k}$ belongs to $B_j$.
Recall that this $B_j$ is a piece of the partition of $X_0^{>k}$ obtained in the process of approximating the 1-cocycle $\sigma_k$ and hence depends on $k$.
Let $b_\omega \in C$ denote the \textit{base at the $Y_k^j$-coordinate} of $\omega$.
This precisely means the following:
The point $\omega \in \Omega$ is encoded by its projection to $X$ and elements $y_k^j\in Y_k^j$ having indices $k\in \N$ and $j=1,\ldots, m(k)$.
The element $y_k^j$ is moreover encoded by elements of $K=C$ indexed by left cosets of $\cals_k^j$ in $(\calr^{\ls k})^y$, where $y$ denotes the projection of $\omega$ to $X^{\ls k}$.
Let us call the element of $C$ indexed by the coset containing the unit $e_y\in (\calr^{\ls k})^y$ the \textit{base at the $Y_k^j$-coordinate} of $\omega$.
We set
\begin{equation}\label{u}
U_k(\omega)=\tilde{V}_{n(k)}(\omega)b_\omega.
\end{equation}

\begin{lem}\label{lem-comm}
Fix $k\in \N$ and $j=1,\ldots, m(k)$.
Pick $g\in \cals_k^j$ and $x\in B_j$, and let $y\in X^{\ls k}$ be the source of $g$.
We put $z=(x, y)\in X$ and $x'=V_{n(k)}x\in X_0^{>k}$.
Pick a point $\omega \in \Omega$ whose projection to $X$ is $z$ and put $\omega'=(e_x, g)\omega$.
Suppose that the equation $\sigma_k(e_x, g)=\sigma_k^j(g)$ holds.
Then we have the equation
\begin{equation}\label{su}
\bar{\fs}(e_x, g)^{-1}U_k(\omega')^{-1}\bar{\fs}(e_{x'}, g)U_k(\omega)=e.
\end{equation}
\end{lem}

Recall that the map $\bar{\fs}\colon \calr \to G$ is the composition of the section $\fs \colon \calr \to G\ltimes X$ and the projection from $G\ltimes X$ onto $G$.
Note that if $g=(V_iy, y)$ or $(T_iy, y)$ for some $i=1,\ldots, k$, then by inequality (\ref{d}), the equation $\sigma_k(e_x, g)=\sigma_k^j(g)$ holds for any $z$ outside a subset of $X$ with small measure.

\begin{proof}[Proof of Lemma \ref{lem-comm}]
Let $b, b'\in C$ be the bases at the $Y_k^j$-coordinates of $\omega$ and $\omega'$, respectively.
Recall that $Y_k^j$ is the space on which $\calr^{\ls k}$ acts, and this action is co-induced from the action $\cals_k^j$ on $C$ defined by the 1-cocycle $\sigma_k^j\colon \cals_k^j\to C$.
The equivalence relation $\calr$ acts on the space $X\times_{X^{\ls k}} Y_k^j$ through the projection onto $\calr^{\ls k}$.
The equation $\omega'=(e_x, g)\omega$ therefore implies that $b'=\sigma_k^j(g)b$.
The equation
\begin{align*}
&\bar{\fs}(e_x, g)^{-1}U_k(\omega')^{-1}\bar{\fs}(e_{x'}, g)U_k(\omega)\\
= \ & \bar{\fs}(e_x, g)^{-1}\tilde{V}_{n(k)}(\omega')^{-1}\bar{\fs}(e_{x'}, g)\tilde{V}_{n(k)}(\omega){b'}^{-1}b=\sigma_k(e_x, g){b'}^{-1}b=\sigma_k^j(g){b'}^{-1}b=e
\end{align*}
then holds, where the second equation follows from the definition of $\sigma_k$.
\end{proof}

We are now ready to check the Jones-Schmidt condition for the sequence $(U_k)_k$.

\begin{lem}\label{lem-c-finite}
With the above notation, the following assertions hold:
\begin{enumerate}
\item[(i)] For any measurable subset $A\subset \Omega$, we have $\eta(U_k^\circ A\bigtriangleup A)\to 0$ as $k\to \infty$.
\item[(ii)] The sequence $(U_k)_k$ asymptotically commutes with any element of $[\calg]$.
\item[(iii)] There exists an a.i.\ sequence $(A_k)_k$ for $\calg$ such that $U_k^\circ A_k=\Omega \setminus A_k$ for any $k\in \N$.
\end{enumerate}
The sequence $(U_k)_k$ is therefore a non-trivial a.c.\ sequence for $\calg$.
\end{lem}

\begin{proof}
We first show assertions (i) and (iii).
These two assertions almost follow from the construction of the action $\Gamma \c \Omega$.
Fix $l\in \N$ and $j=1,\ldots, m(l)$.
For any $k>l$, we have $n(k)>l$ and the transformation $U_k^\circ=\tilde{V}_{n(k)}^\circ$ on $\Omega$ fixes coordinate in $Y_l^j$ because $Y_l^j$ is the space on which $\calr$ acts through the projection onto $\calr^{\ls l}$, and $V_{n(k)}$ is in the full group of the kernel of that projection.
It follows that $U_k^\circ$ acts on the fibered product $X\times_{X^{\ls l}}Y_l^j$ by only exchanging the coordinate indexed by $n(k)$ in $X_0=\prod_\N \Z /2\Z$.
Assertion (i) follows.

For $k\in \N$, we define a subset $A_k\subset \Omega$ as the set of points of $\Omega$ whose projection to $X_0$ has $0$ in the coordinate indexed by $n(k)$.
We have the projection from $\calg$ onto $\calr$ that induces the projection from $\Omega$ onto $X$ between their unit spaces.
By its definition, the sequence $(A_k)_k$ is the inverse image of an a.i.\ sequence for $\calr$, and it is therefore a.i.\ for $\calr$.
As mentioned in the proof of assertion (i), the transformation $U_k^\circ =\tilde{V}_{n(k)}^\circ$ acts on $\Omega$ by exchanging the coordinate indexed by $n(k)$ in $X$.
The equation $U_k^\circ A_k=\Omega \setminus A_k$ therefore holds.
Assertion (iii) follows.

Finally, we show assertion (ii).
The inequalities and lemmas we prepared will be used in this proof.
For each $n\in \N$, we originally define $V_n$ as an element of $[\calr_0\times \cali_1]$ and obtain an element $\tilde{V}_n$ of $[\calg]$ by using the section $\fs \colon \calr \to \calg$.
Similarly we associate to the element $T_n\in [\cali_0\times \calr_1]$ the element of $[\calg]$, $\tilde{T}_n\colon \Omega \to G$, defined so that for a point $\omega \in \Omega$ whose projection to $X$ is $z$, we have $\tilde{T}_n(\omega)=\bar{\fs}(T_nz, z)$.
We also have $\tilde{V}_n(\omega)=\bar{\fs}(V_nz, z)$.

Fix $k\in \N$ and $i=1,\ldots, k$.
We put $W=\tilde{V}_i$ or $\tilde{T}_i$. 
Let $D$ be the set in inequality (\ref{d}).
We claim that $(W\cdot U_k)\omega =(U_k\cdot W)\omega$ for any $\omega \in \Omega$ whose projection to $X$ is in $D$.
Pick such a point $\omega \in \Omega$ and let $z=(x, y)\in D$ denote the projection to $X$.
We have a unique $j=1,\ldots, m(k)$ with $x\in B_j$. 
The condition $z\in D$ implies that $g:=(Wy, y)\in \cals_k^j$ and $\sigma_k(e_x, g)=\sigma_k^j(g)$.
We put $x'=V_{n(k)}x\in X_0^{>k}$ and $\omega'=(e_x, g)\omega \in \Omega$.
By the definition of $\tilde{V}_i$ and $\tilde{T}_i$, we have $W(\omega)=\bar{\fs}(e_x, g)$ and $W(U_k^\circ \omega)=\bar{\fs}(e_{x'}, g)$.
The latter equation holds because the projection of $U_k^\circ \omega \in \Omega$ to $X$ is $(x', y)$.
Equation (\ref{su}) in Lemma \ref{lem-comm} says that $W(U_k^\circ \omega)U_k(\omega)=U_k(\omega')W(\omega)$, and our claim follows.
Assertion (ii) follows from the shown claim, inequality (\ref{d}) and Lemma \ref{lem-cent}.
\end{proof}


\noindent \textbf{The case where $C$ is infinite.} 
Recall that the central subgroup $C$ is embedded into a compact group $K$ with dense image.
As well as in the case where $C$ is finite, for each $k\in \N$, we define $U_k\in [\calg]$ by twisting $\tilde{V}_{n(k)}$ through appropriate multiplication by elements of $C$ depending on points of $\Omega$.
Note that we must not multiply $\tilde{V}_{n(k)}$ by general elements of $K$ because $K$ is not in $G$.
To realize the definition of $U_k$ close to equation (\ref{u}), we use that the equivalence relation given by left multiplication of $C$ on $K$ is hyperfinite.
This approximation process is the only additional task in the case where $C$ is infinite.

Fix $k\in \N$.
We chose the finite subset $F\subset C^{2k}$ in inequality (\ref{x_c}).
Let $\bar{F}\subset C$ be the set of all elements in some coordinate of some element of $F$.
Let $\mu_K$ be the Haar measure on $K$, and let $C$ act on $(K, \mu_K)$ by left multiplication.
We denote by $\calr_{C, K}$ the equivalence relation associated with this action.
Since $C$ is abelian and $\calr_{C, K}$ is thus hyperfinite (\cite{dye1}), there exists a finite subrelation $\calt$ of $\calr_{C, K}$ such that
\begin{equation}\label{l}
\mu_K(L)>1-\varepsilon_k/m(k),\ {\rm where\ we\ set}\ L=\{ \, b\in K\mid \forall c\in \bar{F},\ (cb, b)\in \calt \, \}.
\end{equation}
Let $\Sigma \subset K$ be a fundamental domain of $\calt$, namely, $\Sigma$ is a measurable subset of $K$ such that any equivalence class of $\calt$ intersects $\Sigma$ at exactly one point.
We note that for any $b\in K$, there exists a unique $c\in C$ such that $c^{-1}b\in \Sigma$ and $(c^{-1}b, b)\in \calt$.

We define $U_k\in [\calg]$ as follows:
Pick a point $\omega \in \Omega$ and $j=1,\ldots, m(k)$ such that the projection of $\omega$ to $X_0^{>k}$ belongs to $B_j$.
Let $b=b_\omega \in K$ be the base at the $Y_k^j$-coordinate of $\omega$.
Find a unique $c=c_b\in C$ such that $c^{-1}b\in \Sigma$ and $(c^{-1}b, b)\in \calt$.
We set
\begin{equation}\label{u-c-infinite}
U_k(\omega)=\tilde{V}_{n(k)}(\omega)c.
\end{equation}
Adding an assumption on a point of $\Omega$, we obtain the following similar to Lemma \ref{lem-comm}.

\begin{lem}\label{lem-comm-2}
Fix $k\in \N$ and $j=1,\ldots, m(k)$.
Pick $g\in \cals_k^j$ and $x\in B_j$, and let $y\in X^{\ls k}$ be the source of $g$.
We put $z=(x, y)\in X$ and $x'=V_{n(k)}x\in X_0^{>k}$.
Pick a point $\omega \in \Omega$ whose projection to $X$ is $z$ and put $\omega'=(e_x, g)\omega$.
Let $b\in K$ be the base at the $Y_k^j$-coordinate of $\omega$.
Suppose that the equation $\sigma_k(e_x, g)=\sigma_k^j(g)$ holds and that $(\sigma_k^j(g)b, b)\in \calt$.
Then we have the equation
\begin{equation}\label{su2}
\bar{\fs}(e_x, g)^{-1}U_k(\omega')^{-1}\bar{\fs}(e_{x'}, g)U_k(\omega)=e.
\end{equation}
\end{lem}

\begin{proof}
Let $b'\in K$ be the base at the $Y_k^j$-coordinate of $\omega'$.
We have $\sigma_k^j(g)b=b'$ and thus $(b', b)\in \calt$.
Put $c=c_b$ and $c'=c_{b'}$.
The condition $(c^{-1}b, b), ({c'}^{-1}b', b'), (b', b)\in \calt$ implies $(c^{-1}b, {c'}^{-1}b')\in \calt$.
The points $c^{-1}b$ and ${c'}^{-1}b'$ belong to $\Sigma$, and thus $c^{-1}b={c'}^{-1}b'$ because $\Sigma$ is a fundamental domain of $\calt$.
The equation
\begin{align*}
&\bar{\fs}(e_x, g)^{-1}U_k(\omega')^{-1}\bar{\fs}(e_{x'}, g)U_k(\omega)\\
= \ & \bar{\fs}(e_x, g)^{-1}\tilde{V}_{n(k)}(\omega')^{-1}\bar{\fs}(e_{x'}, g)\tilde{V}_{n(k)}(\omega){c'}^{-1}c=\sigma_k(e_x, g){b'}^{-1}b=\sigma_k^j(g){b'}^{-1}b=e
\end{align*}
therefore holds.
\end{proof}

For $k\in \N$, we define a subset $\Omega_k\subset \Omega$ as the set of points $\omega$ such that the projection of $\omega$ to $X$ is in $D$ and the base of the $Y_k^j$-coordinate of $\omega$ is in $L$ for any $j=1,\ldots, m(k)$.
This is a measurable subset of $\Omega_k$ with
\begin{equation}\label{omega_k}
\eta(\Omega_k)\geq \mu(D)\mu_K(L)^{m(k)}>(1-8\varepsilon_k)(1-\varepsilon_k)>1-9\varepsilon_k,
\end{equation}
where the second inequality follows from inequalities (\ref{d}) and (\ref{l}).

\begin{lem}\label{lem-c-infinite}
For the map $U_k$ in equation (\ref{u-c-infinite}), the sequence $(U_k)_k$ satisfies the same properties (i)--(iii) as those in Lemma \ref{lem-c-finite}.
The sequence $(U_k)_k$ is therefore a non-trivial a.c.\ sequence for $\calg$.
\end{lem}

\begin{proof}
For assertions (i) and (iii), the same proof as those in Lemma \ref{lem-c-finite} is available.
We show assertion (ii): The sequence $(U_k)_k$ asymptotically commutes with any element of $[\calg]$.

Fix $k\in \N$ and $i=1,\ldots, k$.
Put $W=\tilde{V}_i$ or $\tilde{T}_i$.
We claim that $(W\cdot U_k)\omega =(U_k\cdot W)\omega$ for any $\omega \in \Omega_k$. 
By inequality (\ref{omega_k}) and Lemma \ref{lem-cent}, this ends the proof of assertion (ii).
Pick $\omega \in \Omega_k$ and let $z=(x, y)\in D$ be the projection of $\omega$ to $X$.
We have a unique $j=1,\ldots, m(k)$ with $x\in B_j$.
Let $b\in K$ be the base at the $Y_k^j$-coordinate of $\omega$, which is in $L$ because $\omega \in \Omega_k$.
It follows from $z\in D$ that we have $g:=(Wy, y)\in \cals_k^j$, $\sigma_k(e_x, g)=\sigma_k^j(g)$ and $z\in X_c$ for some $c\in F$, and therefore have $\sigma_k^j(g)\in \bar{F}$.
It follows from $b\in L$ that $(\sigma_k^j(g)b, b)\in \calt$.
We now apply Lemma \ref{lem-comm-2}.
Through equation (\ref{su2}), the claim is proved along a verbatim translation of the proof of Lemma \ref{lem-c-finite} (ii). 
\end{proof}

By Lemmas \ref{lem-c-finite} and \ref{lem-c-infinite}, the groupoid $\calg =G\ltimes (\Omega, \eta)$ has a non-trivial a.c.\ sequence, no matter whether $C$ is finite or not.
We now assume that the action $\Gamma \c (X, \mu)$ is ergodic, and deduce Theorem \ref{thm-ce}.
Let $P\colon \Omega \to X$ denote the projection.
Let $\theta \colon (\Omega, \eta)\to (T, \xi)$ be the ergodic decomposition for the action $\Gamma \c (\Omega, \eta)$ and $\eta =\int_T \eta_t \, d\xi(t)$ the disintegration of $\eta$ with respect to $\theta$.
For $\xi$-a.e.\ $t\in T$, the measure $P_*\eta_t$ on $X$ is equal to $\mu$.
This claim is proved as follows:
By \cite[Theorem 3.2]{v}, we may assume that $X$ is a compact Polish space and $\Gamma$ acts on $X$ by homeomorphisms.
Pushing the measure $\xi$ out through the map sending $t\in T$ to $P_*\eta_t$, we obtain the probability measure $\bar{\xi}$ on the space of $\Gamma$-invariant probability measures on $X$, whose barycenter is $\mu$ because $P_*\eta =\mu$.
The measures $\mu$ and $P_*\eta_t$ with $\xi$-a.e.\ $t\in T$ are extreme points in that space, and by Bauer's characterization of extreme points (\cite[Proposition 1.4]{p}), $\bar{\xi}$ is supported on the single point $\mu$.
The claim follows.
It implies that for $\xi$-a.e.\ $t\in T$, the action $\Gamma \c (\Omega, \eta_t)$ is an extension of the action $\Gamma \c (X, \mu)$.
Applying Remark \ref{rem-1/2} and putting $(Z, \zeta)=(\Omega, \eta_t)$ for some $t\in T$, we obtain Theorem \ref{thm-ce}.




\end{document}